\documentclass[11pt,a4paper]{amsart}
\usepackage[margin=3cm]{geometry}
\usepackage{color}                    

\usepackage[pdftex,hyperfootnotes=false,colorlinks=true,linkcolor=blue,
citecolor=purple,filecolor=magenta,urlcolor=cyan,hypertexnames=false]{hyperref}

\usepackage{anyfontsize,enumerate,stmaryrd,mathtools,thmtools,tikz,txfonts,nccmath}
\usepackage[oldstyle]{libertine}%
\usepackage[T1]{fontenc}
\usepackage[utf8]{inputenc}
\usepackage{csquotes}
\makeatletter
\renewcommand*\libertine@figurestyle{LF}
\makeatother
\usepackage[libertine,libaltvw,liby]{newtxmath}
\makeatletter
\renewcommand*\libertine@figurestyle{OsF}
\makeatother
\usepackage[style=alphabetic,maxalphanames=5,giveninits,sorting=nyt,%
maxbibnames=99]{biblatex}
\usepackage{tikz-cd}
\usepackage[noabbrev]{cleveref}
\usepackage{autonum}
\crefname{lemma}{lemma}{lemmata}
\Crefname{lemma}{Lemma}{Lemmata}
\crefname{subsection}{subsection}{subsections}
\Crefname{subsection}{Subsection}{Subsections}

\newtheorem{theorem}{Theorem}[section]
\newtheorem{lemma}[theorem]{Lemma}
\newtheorem{proposition}[theorem]{Proposition}
\newtheorem{corollary}[theorem]{Corollary}

\newtheorem*{fact}{Fact}
\theoremstyle{definition}
\newtheorem{definition}[theorem]{Definition}
\newtheorem{remark}[theorem]{Remark}
\newtheorem{algorithm}[theorem]{Algorithm}

\newtheorem{example}[theorem]{Example}
\newtheorem{construction}[theorem]{Construction}

\begingroup\expandafter\expandafter\expandafter\endgroup
\expandafter\ifx\csname pdfsuppresswarningpagegroup\endcsname\relax
\else
\fi
{}

\title{Bi-pruned Hurwitz numbers}
\author[M.~A.~Hahn]{Marvin Anas Hahn}
\address{M.~A.~Hahn: Mathematisches Institut, Universit\"at T\"ubingen, Auf der Morgenstelle 10, 72076 T\"ubingen, Germany.}
\email{marvin-anas.hahn@uni-tuebingen.de}


\begin{document}
\begin{abstract}
Hurwitz numbers enumerate ramified coverings of the Riemann sphere with fixed ramification data. Certain kinds of ramification data are of particular interest, such as double Hurwitz numbers, which count covers with fixed arbitrary ramification over $0$ and $\infty$ and simple ramification over $b$ points, where $b$ is given by the Riemann-Hurwitz formula. In this work, we introduce the notion of \textit{bi-pruned double Hurwitz numbers}. This is a new enumerative problem, which yields smaller numbers but completely determines double Hurwitz numbers. They count a relevant subset of covers and share many properties with double Hurwitz numbers, such as piecewise polynomial behaviour and an expression in the symmetric group. Thus, we may view them as a core of the double Hurwitz numbers problem. This work is built on and generalises previous work of Do--Norbury \cite{DNpruned} and the author \cite{Hahnpruned}.
\end{abstract}

\maketitle

\section{Introduction}
Hurwitz numbers were introduced in the late 19th century by A. Hurwitz in \cite{Hurwitzverzweigung} as an enumeration of ramified coverings of the Riemann sphere with fixed ramification. These invariants are related to several branches of mathematics, such as algebraic geometry, algebraic topology, tropical geometry, representation theory of the symmetric group, operator theory, free probability theory and more. Hurwitz numbers have many variants and satisfy definitions in several different settings. Some variants have turned out to be of greater importance than others. One of the more important Hurwitz numbers are those initially studied by Hurwitz, which are now called \textit{single Hurwitz numbers} and their generalisation by Okounkov called \textit{double Hurwitz numbers} \cite{Okounkovdouble}.\par
\textbf{Single Hurwitz numbers.} Single Hurwitz numbers count those ramified covers with arbitrary ramification over $0$ and simple ramification over $b$ other points given by the Riemann-Hurwitz formula. They satisfy many fascinating properties, among which is a cut-and-join recursion and polynomial behaviour up to a combinatorial pre-factor, the latter being a direct consequence of the celebrated ELSV formula, which expresses single Hurwitz numbers as intersection numbers on the moduli space of stable curves with marked points $\overline{\mathcal{M}}_{g,n}$. Moreover, single Hurwitz numbers are connected to the powerful Chekhov-Eynard-Orantin (CEO) topological recursion. CEO topological recursion is a formalism, which associates to a spectral curve a family of differentials on a Riemann surface. Those differentials satisfy a recursive structure. For many enumerative invariants, there exist spectral curves, such that the associated differentials encode these invariants as coefficients of local expansion. This is true for single Hurwitz numbers as well. One also says \textit{single Hurwitz numbers satisfy CEO topological recrusion} \cite{EOtopological}.\par 
\textbf{Double Hurwitz numbers.} Double Hurwitz numbers count ramified covers with arbitrary ramification over $0$ and $\infty$ and simple ramification over $b$ points given by the Riemann-Hurwitz formula. They satisfy many properties parallel to the ones for single Hurwitz numbers, such as a cut-and-join recursion and piecewise polynomial behaviour \cite{GJtransitive,GJVdouble}. Progress towards topological recursion has been made in \cite{ACEHweightedhurwitz}.\par
\textbf{Pruned Hurwitz numbers.} In \cite{DNpruned}, it was proved that one can reduce the single Hurwitz numbers computation to a subset of the involved covers. This new enumerative problem ist called \textit{pruned single Hurwitz number} and determines the usual single Hurwitz number completely. Moreover, it was proved by Do--Norbury, that the pruned single Hurwitz numbers satisfy a cut-and-join recursion, polynomial behaviour, an interpretation in the symmetric group, an interpretation in terms of intersection numbers and topological recursion. In order to define pruned single Hurwitz numbers, one considers a graph theoretic interpretation of single Hurwitz numbers in terms of so-called \textit{branching graphs} \cite{OPhurwitz}. There is a bijection branching graphs and covers contributing to the single Hurwitz numbers. In these graphs, faces correspond to pre-images of $\infty$, vertices to pre-images of $0$ and edges to simple branch points. One defines pruned single Hurwitz numbers as the sum over those covers, which correspond to branching graphs without leaves, i.e. without $1-$valent vertices. Moreover, one obtains the correspondence between single Hurwitz and pruned single Hurwitz numbers by consecutively removing leaves of branching graphs and examining the combinatorics of this process. This process of removing leaves is called the \textit{pruning process}. In \cite{Hahnpruned}, the notion of \textit{pruned double Hurwitz numbers} was introduced, which generalises the notion of pruned single Hurwitz numbers. These invariants are defined analgously by considering an analog of branching graphs suitable for double Hurwitz numbers. Once again, the pruning process, i.e. consecutively removing leaves yields a correspondence result, i.e. pruned double Hurwitz numbers completely determine double Hurwitz numbers and vice versa. Moreover, parallel to the single Hurwitz numbers case, pruned double Hurwitz numbers share many structural properties with their usual counterparts. In particular, they satisfy piecewise polynomial behaviour, a cut-and-join recursion and an interpretation in the symmetric group.\par
In this work, we introduce the notion of \textit{bi-pruned double Hurwitz numbers}. Instead of considering branching graphs corresponding to covers contributing to the double Hurwitz numbers, we focus on a different class of graphs -- so-called \textit{Hurwitz galaxies}. As for branching graphs, there exists a bijection between covers contributing to the double Hurwitz number and Hurwitz galaxies \cite{Johnsontropicalization}. These graphs consist of bi-coloured faces and encode covers as follows: White faces correspond to pre-images of $0$, black faces to pre-images of $\infty$, $2-$valent vertices to unramified pre-images and $4-$valent vertices to ramified pre-images of simple branch points.\par 
In this setting, the pruning process corresponds to removing certain white faces, which we call \textit{loop faces}. The change of perspective by considering Hurwitz galaxies instead of branching graphs now brings the advantage that we can also remove black loop faces, i.e. pre-images of $\infty$. By consecutively removing loop faces corresponding to pre-images of $0$ and $\infty$ one obtains significantly smaller graphs in comparison to \cite{Hahnpruned}. This is due to the fact that removing white loop faces creates new black loop faces and vice versa.\par 
We prove a correspondence theorem, expressing double Hurwitz numbers in terms of bi-pruned double Hurwitz numbers. Further, we prove that bi-pruned double Hurwitz numbers are piecewise polynomial and admit an interpretation in the symmetric group. This follows the philosophy of \cite{DNpruned} and \cite{Hahnpruned} defining the core of the Hurwitz numbers problem and generalises the results from \cite{Hahnpruned}.

\subsection{Structure of the paper}
In \cref{sec:pre}, we introduce the relevant basic notions concerning Hurwitz numbers and Hurwitz galaxies, which are our main technical tool. We continue in \cref{sec:main} by defining bi-pruned Hurwitz numbers and introducing the necessary tools to state our main theorem. We further make two structural observations in \cref{sec:struc}. Lastly, we prove our main theorem in \cref{sec:proof}.

\subsection{Acknowledgements}
The author thanks Hannah Markwig for her careful proofreading and useful comments. Further, the author is grateful for interesting discussions with Maxim Karev and Felix Leid. Many computations for this project have been made using GAP \cite{GAP}. The author gratefully acknowledges partial support by DFG SFB-TRR 195 Symbolic tools in mathematics and their applications, project A 14 Random matrices and Hurwitz numbers (INST 248/238-1).

\section{Preliminaries}
\label{sec:pre}
Before we recall the relevant basic notions, we introduce the notation $[m]\coloneq\{1,\dots,m\}$ for $m\in\mathbb{Z}_{\ge1}$. For a more in-depth introduction to Hurwitz numbers, we recommend \cite{CMfirstcourse}.

\begin{definition}
Let $\mu$ and $\nu$ be partitions of the same positive integer $d$, let $g$ be a non-negative integer and denote $b=2g-2+\ell(\mu)+\ell(\nu)$. We define a Hurwitz cover of type $(g,\mu,\nu)$ as a morphism $f:S\to \mathbb{P}^1$, such that
\begin{enumerate}[(i)]
\item $f$ is of degree $d$,
\item $S$ is a surface of genus $g$,
\item $f$ ramifies with profile $\mu$ over $0$ and profile $\nu$ of $\infty$,
\item the preimages of $0$ (resp. $\infty$) are labelled by $1,\dots,\ell(\mu)$ (resp. $1,\dots,\ell(\nu)$),
\item $f$ ramifies with profile $(2,1,\dots,1)$ over the $b-$th roots of unity.
\end{enumerate}
We call two Hurwitz covers $f_1:S_1\to\mathbb{P}^1$ and $f_2:S_2\to\mathbb{P}^1$ of type $(g,\mu,\nu)$ isomorphic, if there exists a homeomorphism $h:S_1\to S_2$, such that $f_1=f_2\circ h$. Then we define the double Hurwitz number associated to the above data by
\begin{equation}
h_g(\mu,\nu)=\sum \frac{1}{|\mathrm{Aut}(f)|},
\end{equation}
where we some over all equivalence classes of Hurwitz covers $f$ of type $(g,\mu,\nu)$.
\end{definition}

These numbers can be counted in terms of so-called \textit{Hurwitz galaxies}. These are obtained as follows: We fix the data $\mu,\nu,g$ as before and consider the graph $\Gamma_b$ whose vertices are the $b-$th roots of unity and whose edges form a cycle along the unit circle. For each Hurwitz cover $f:S\to\mathbb{P}^1$ of type $(g,\mu,\nu)$, we obtain a graph on $S$ by pulling back $\Gamma_b$ along $f$. The following definition is a characterisation of these graphs.

\begin{definition}
 \label{def:hurwitzgalaxy}
 A Hurwitz galaxy of type $(g,\mu,\nu)$ is a graph $G$ on an oriented surface $S$ of genus $g$, such that for $b=2g-2+\ell(\mu)+\ell(\nu)$:
 \begin{enumerate}[(i)]
  \item $S\backslash G$ is a disjoint union of open disks,
  \item $G$ partitions $S$ into $\ell(\mu)+\ell(\nu)$ disjoint faces,
  \item these faces are coloured black and white, such that $\ell(\nu)$ many faces are coloured black and $\ell(\mu)$ many faces are coloured white, such that each edge is incident to a white face on one side and to a black face on the other side,
  \item the white (resp. black) faces are labelled by $1,\dots,\ell(\mu)$ (resp. $1,\dots,\ell(\nu)$), such that a face labelled $i$ (resp. $j$) is bounded by $\mu_i\cdot b$ (resp. $\nu_j\cdot b$) vertices and we call $\mu_i$ (resp. $\nu_j$) the \textit{perimeter} of the face labelled $i$ (resp. $j$),
  \item the vertices in the boundary of white faces are labelled cyclically counterclockwise by $1,\dots,b$ (this implies that the vertices in the boundary of black faces are labelled cyclically clockwise by $1,\dots,b$),
  \item for each $i\in\{1,\dots,b\}$, there are $d-2$ vertices labelled $i$, which are 2-valent and one vertex labelled $i$, which is 4-valent.
 \end{enumerate}
 An isomorphism of Hurwitz galaxies is a homeomorphism of the underlying surfaces which restricts to a graph isomorphism respecting all labellings.
 \end{definition}
 
As mentioned above Hurwitz numbers can be computed in terms of Hurwitz galaxies.
 
\begin{theorem}[\cite{OPhurwitz}]
Let $g,\mu,\nu$ be data as before, then
\begin{equation}
h_g(\mu,\nu)=\sum \frac{1}{|\mathrm{Aut}(G)|},
\end{equation} 
where we sum over all isomorphism classes of Hurwitz galaxies of type $(g,\mu,\nu)$.
 \end{theorem}
 
\begin{remark}
Most Hurwitz galaxies do not have automorphisms, in fact the only automorphisms appear for $\ell(\mu)+\ell(\nu)=2$.
\end{remark}

There is a further notion of expressing Hurwitz numbers in terms of weighted graph counts. The notion involved are so-called \textit{branching graphs}. The transition from Hurwitz galaxies to branching graphs is done as follows (see also Figure 3 in \cite{Hahnpruned}):

\begin{enumerate}
\item We start with a Hurwitz galaxy of type $(g,\ell(\mu),\ell(\nu))$.
\item We draw a vertex for each white face and label it by same integer as the face.
\item For each new vertex labelled $i$, we draw edges connecting the vertex $i$ to the vertices in the boundary of the corresponding white face. We label the edge connecting $i$ to a vertex labelled $k\in[b]$ by $k$ and thus obtain $\mu_i\cdot b$ edges at $i$ cyclically labelled by $1,\dots, b$.
\item We now remove the vertices and edges of the old Hurwitz galaxy.
\item We obtain a new graph on $[\ell(\mu)]$ many vertices. We obtain a half-edge for each $2-$valent vertex. As each $4-$valent vertex is adjacent to two edges in our construction, we obtain one edge for each $4-$valent vertex by removing the vertex in the Hurwitz galaxy.
\item The graph we obtain is \textit{branching graph of type} $(g,\mu,\nu)$.
\end{enumerate}

This process is in fact a bijection (Proposition 9 \cite{Hahnpruned}), in the sense that each Hurwitz galaxy yields a unique branching graph and vice versa and their automorphism groups are isomorphic. The following is a characterisation of branching graphs.
 \begin{definition}
\label{def:branchinggraph}
 Let $d$ be a positive integers, $\mu$ and $\nu$ be ordered partitions of $d$. We define a branching graph of type $(g,\mu,\nu)$ to be a graph $\varGamma$ on an oriented surface $S$ of genus g, such that for $b=\ell(\mu)+\ell(\nu)-2+2g$:
 \begin{enumerate}[(i)]
 \item $S\backslash\varGamma$ is a disjoint union of open disks.
  \item There are $\ell(\mu)$ vertices, labeled $1,\dots,\ell(\mu)$, such that the vertex labeled $i$ is adjacent to $\mu_i\cdot m$ half-edges, labeled cyclically counterclockwise by $1,\dots,m$. We define the perimeter of the vertex labeled $i$ by $per(i)=\mu_i$.
  \item There are exactly $m$ full edges labeled by $1,\dots,b$.
  \item The $\ell(\nu)$ faces are labeled by $1,\dots,\ell(\nu)$ and the face labeled $i$ has perimeter $per(i)=\nu_i$, by which we mean, that each label occurs $\nu_i$ times inside the corresponding face, where we count full-edges adjacent to $i$ on both sides twice.
 \end{enumerate}
Note, that we allow loops at the vertices.  An isomorphism of branching graphs is a homeomorphism of the underlying surfaces which restricts to a graph isomorphism respecting all labellings.\par 
For a fixed branching graph $\varGamma$, we obtain its \textit{underlyind reduced branching graph} by removing all half-edges.
\end{definition}

The above discussion can be summarised in the following corollary.

\begin{corollary}
Let $g,\mu,\nu$ be data as before, then
\begin{equation}
h_g(\mu,\nu)=\sum \frac{1}{|\mathrm{Aut}(\varGamma)|},
\end{equation} 
where we sum over all equivalence classes of branching graphs of type $(g,\mu,\nu)$.
\end{corollary}

\begin{remark}
We note, that branching graphs can be obtained by pulling back a graph on $\mathbb{P}^1$ similar to the situation with Hurwitz galaxies. The graph $G_b$ we pull back is given by vertices at $0$ and the $b-th$ roots of unity and the edges connect $0$ with each $b-$th root of unity in a straight path. For a fixed cover $f:S\to\mathbb{P}^1$ of type $(g,\mu,\nu)$, we consider $f^{-1}(G_b)$, and draw a vertex at each of the preimages of $0$. We further label vertices and faces according to the labels of the preimages of $0$ and $\infty$ respectively. We obtain a branching graph of type $(g,\mu,\nu)$ on $S$. By observing that branching graphs of type $(g,\mu,\nu)$ are in bijection to the equivalence classes of covers of type $(g,\mu,\nu)$ and the automorphism group of the branching graph is isomorphic to the automorphism group of the cover one obtains the above corollary without the detour via Hurwitz galaxies.
\end{remark}

We finish this section with the definition of pruned double Hurwitz numbers in \cite{DNpruned,Hahnpruned}.

\begin{definition}
Let $g$ be a non-negative integer and let $\mu,\nu$ be partitions of the same integer. Then we define \textit{pruned double Hurwitz numbers} by
\begin{equation}
\mathcal{PH}_g(\mu,\nu)=\sum_{\varGamma}\frac{1}{|\mathrm{Aut}(\varGamma)|},
\end{equation}
where we sum over all branching graphs of type $(g,\mu,\nu)$ without leaves.
\end{definition}

\section{The bi-pruning correspondence}
\label{sec:main}
We begin by defining the underlying bi-pruned structure of a Hurwitz galaxy.

\begin{definition}
We call a face of a Hurwitz galaxy a \textit{loop face} if it is adjacent to at most one $4-$valent vertex. We call a Hurwitz galaxy without loop faces a \textit{bi-pruned Hurwitz galaxy}.\par 
Let $\mu,\nu$ be partitions of the same positive integer and $S$ a surface of genus $g$. Then we define the \textit{bi-pruned double Hurwitz numbers} by
\begin{equation}
Ph_g(\mu,\nu)=\sum\frac{1}{|\mathrm{Aut}(G)|},
\end{equation}
where we sum over all bi-pruned Hurwitz galaxies of type $(g,\mu,\nu)$. We also define
\begin{equation}
\widehat{Ph}_g(\mu,\nu)=\sum 1,
\end{equation}
where we sum over all bi-pruned Hurwitz galaxies of type $(g,\mu,\nu)$.\par
\end{definition}

\begin{remark}
As there are no automorphisms for $\ell(\mu)+\ell(\nu))\neq2$, we have
\begin{equation}
Ph_g(\mu,\nu)=\widehat{Ph}_g(\mu,\nu)
\end{equation}
in those cases. Thus, we can disregard automorphisms for most of our discussion.
\end{remark}

In our next step, we understand the notion of loop faces in terms of branching graphs.

\begin{proposition}
\label{prop:bibranch}
Let $G$ be a Hurwitz galaxy and let $\Gamma$ be the branching graph corresponding to $G$.
\begin{enumerate}
\item Let $F$ be a black face of $G$ and let $\tilde{F}$ be the face in $\Gamma$ corresponding to $F$. Then $F$ is a loop face if and only if $\tilde{F}$ is a face bounded by a single edge, which is a loop.
\item Let $F$ be a white face of $G$ and let $v_F$ be the vertex in $\Gamma$ corresponding to $F$. Then $F$ is a loop face if and only if $v_F$ is a leaf.
\end{enumerate}
\end{proposition}

\begin{proof}
The first statement is illustrated in \cref{fig:loopblack} and the second statement is illustrated in \cref{fig:loopwhite}. We see in \cref{fig:loopblack} that the black loop face in the Hurwitz galaxy indeed yields a face in the branching graph, which is only bounded by a single edge, which is a loop. Similarly, we see in \cref{fig:loopwhite} that the white loop face in the Hurwitz galaxy yields a leaf in the branching graph.  As the interior of each face is isomorphic to the unit circle, these illustrations in genus $0$ are representative of the general case.
\end{proof}

We obtain the following corollary.

\begin{corollary}
Let $g$ be a non-negative integer and $\mu,\nu$ partitions of the same positive integer, then
\begin{equation}
\mathcal{PH}_g(\mu,\nu)=\sum \frac{1}{|\mathrm{Aut}(G)|},
\end{equation}
where we sum over all Hurwitz galaxies of type $(g,\mu,\nu)$ without white loop faces.
\end{corollary}

\begin{figure}
\begin{center}
\scalebox{0.8}{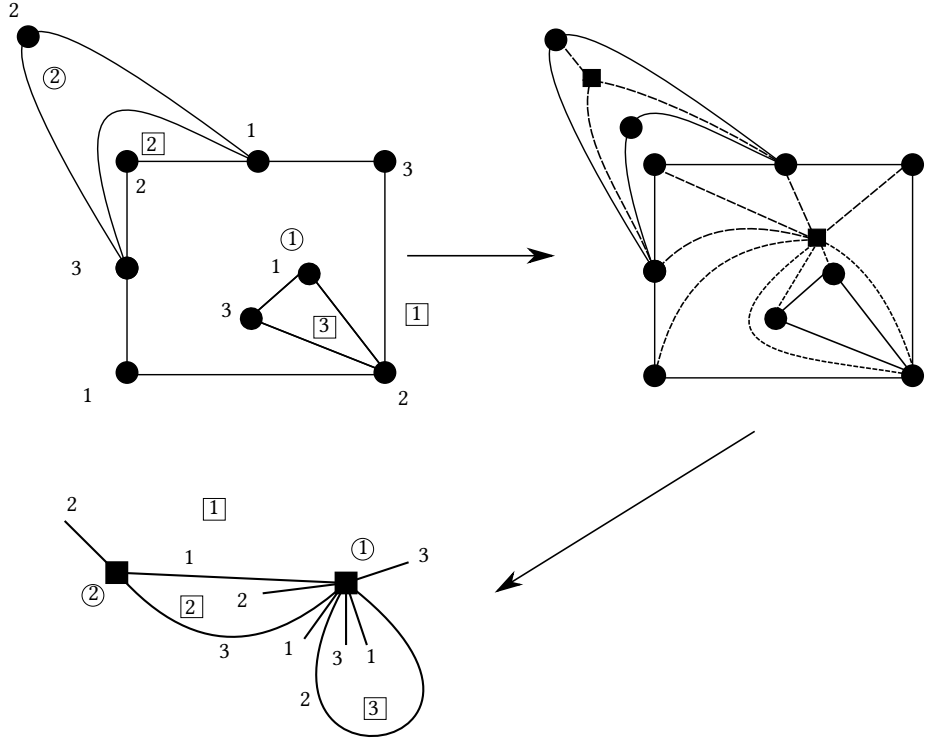}
\caption{We start with a Hurwitz galaxy of type $(0,(3,1),(2,1,1))$, where white faces are indicated by circled face labels and black faces are indicated by squared face labels (top left). Then we draw a vertex for each white face and connect each to the vertices in the boundary of the respective face (top right). Finally, we delete the graph structure of the Hurwitz galaxy and obtain a branching graph of type $(0,(3,1),(2,1,1))$ (bottem left).}
\label{fig:loopblack}
\end{center}
\end{figure}

\begin{figure}
\begin{center}
\scalebox{0.8}{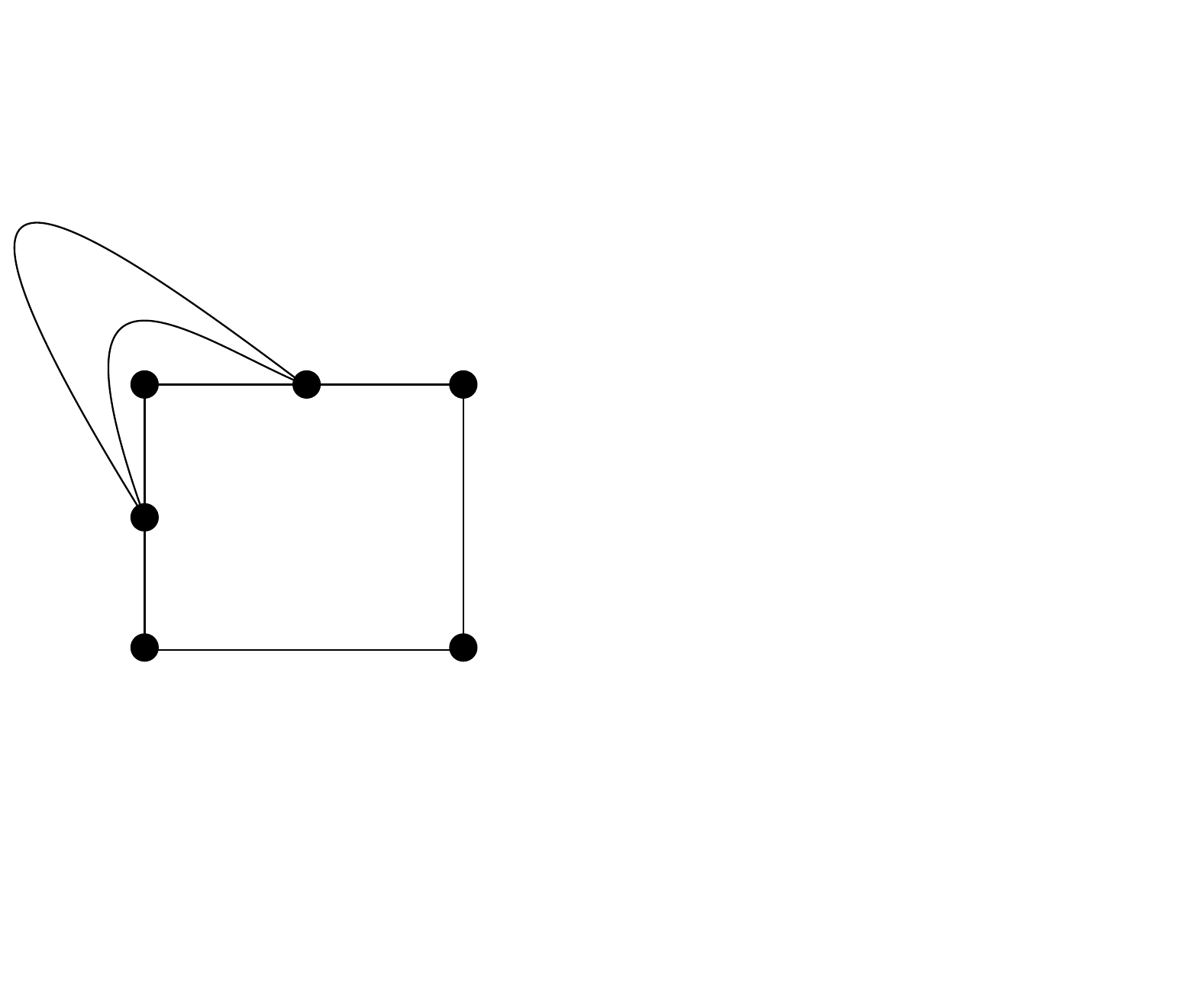}
\caption{We note that in this figure, we reverse the orientation in order to illustrate the process parallel to the one in \cref{fig:loopblack}. We start with a Hurwitz galaxy of type $(0,(2,1,1),(3,1))$, where white faces are indicated by circled face labels and black faces are indicated by squared face labels (top left). Then we draw a vertex for each white face and connect each to the vertices in the boundary of the respective face (top right). Finally, we delete the graph structure of the Hurwitz galaxy and obtain a branching graph of type $(0,(2,1,1),(3,1))$ (bottem left).}
\label{fig:loopwhite}
\end{center}
\end{figure}

We now compute a few examples comparing double Hurwitz number, pruned double Hurwitz numbers and bi-pruned double Hurwitz numbers. Most of our computation were at least supported by GAP \cite{GAP}, as \cref{prop:sym} enables computations of Hurwitz numbers in terms of the symmetric group.

\begin{example}
\label{ex:comp}
We fix $g=0,\mu=(2,2),\nu=(3,1)$. We can compute $h_g(\mu,\nu)$ and $Ph_g(\mu,\nu)$ using the above interpretation in the symmetric group. We obtain
\begin{equation}
h_g(\mu,\nu)=12\quad\textrm{and}\quad Ph_g(\mu,\nu)=2.
\end{equation}
In comparison, for the pruned double Hurwitz numbers in \cite{Hahnpruned}, we obtain $\mathcal{PH}_g(\mu,\nu)=2$.\par
For $g=0,\mu=(2,1),\nu=(1,1,1)$, we obtain
\begin{equation}
h_g(\mu,\nu)=24\quad\textrm{and}\quad Ph_g(\mu,\nu)=6.
\end{equation}
In comparison, for the pruned double Hurwitz numbers in \cite{Hahnpruned}, we obtain $\mathcal{PH}_g(\mu,\nu)=24$.
\end{example}

We introduce the \textit{bi-pruning} process, where we remove white \textbf{and} black faces consecutively, thus generalising the pruning process in \cite{DNpruned,Hahnpruned}.

\begin{construction}
\label{con:prun}
Let $G$ be a Hurwitz galaxy of type $(g,\mu,\nu)$. We construct a new Hurwitz galaxy $G_{prun}$ of type $(g,\mu',\nu')$, where $\ell(\mu')\le\ell(\mu),\ell(\nu')\le\ell(\nu)$ as follows:

\begin{enumerate}
\item If $G$ does not contain a loop face, then $G=G_{prun}$.
\item If $G$ does contain a loop face, we choose one loop face $F$ labelled $j$.
\item The loop face $F$ is adjacent to one $4-$valent vertex $w$ labelled $i$ for $i\in[b]$.
\item We remove all vertices and edges adjacent to $F$ from $G$ except for the vertex labelled $i$.
\item The previously $4-$valent vertex $w$ labelled $i$ is now $2-$valent. Let this vertex be adjacent to two vertices $v, v'$. Then we remove $w$ and its $2$ adjacent edges and joint the vertices $v,v'$ by a new edge.
\item We adjust the face labelling as follows: If $F'$ is a face of $G$ with the same colour as $F$ with a label smaller than $i$, it does not change. If the label of $F'$ is bigger than $i$, we reduce it by $1$.
\item We adjust the vertex labelling by relabelling the vertices with labels in $[b-1]$ but maintaing the linear order the labels.
\item We obtain a graph $G'$.
\item If $G'$ does contain a loop face, we go to step $2$. If $G'$ does not contain a loop face we are done.
\end{enumerate}
We call the resulting graph the \textit{underlying pruned Hurwitz galaxy} $G_{prun}$ of $G$.
\end{construction}

\begin{remark}
Before we continue our discussion, we make three remarks
\begin{itemize}
\item We note that the underlying pruned real Hurwitz galaxy is indeed a Hurwitz galaxy of type $(g,\mu',\nu')$, where $\ell(\mu')\le\ell(\mu),\ell(\nu')\le\ell(\nu)$ and there exists a subpartition $\tilde{\mu}$ (resp. $\tilde{\nu}$) of $\mu$ (resp. $\nu$) of length $\ell(\mu')$ (resp. $\ell(\nu')$), such that all $\mu'\le\tilde{\mu}$ (resp. $\nu'\le\tilde{\nu}$) entrywise.
\item Note, that we may obtain the empty Hurwitz galaxy by removing the entire graph structure.
\item The pruning process in \cref{con:prun} differs from the one in \cite{DNpruned,Hahnpruned} in the sense that in the latter only white loop faces are pruned. The advantage of the approach in \cref{con:prun} is that we may obtain much smaller graphs, as black faces might become loop faces after pruning white loop faces and vice versa (see \cref{ex:prun}).
\end{itemize}
\end{remark}

\begin{example}
\label{ex:prun}
We illustrate \cref{con:prun} on a branching graph of type $(0,(1,6,1,1),(3,2,1,3)$ in \cref{fig:ex}. After the first step, we obtain a Hurwitz galaxy of type $(0,(6,1,1),(2,2,1,3))$. After the second step, we obtain a Hurwitz galaxy of type $(0,(5,1,1),(2,2,3))$. Finally, after the third step we obtain a bi-pruned Hurwitz galaxy of type $(0,(3,1,1),(2,3))$.

\begin{figure}
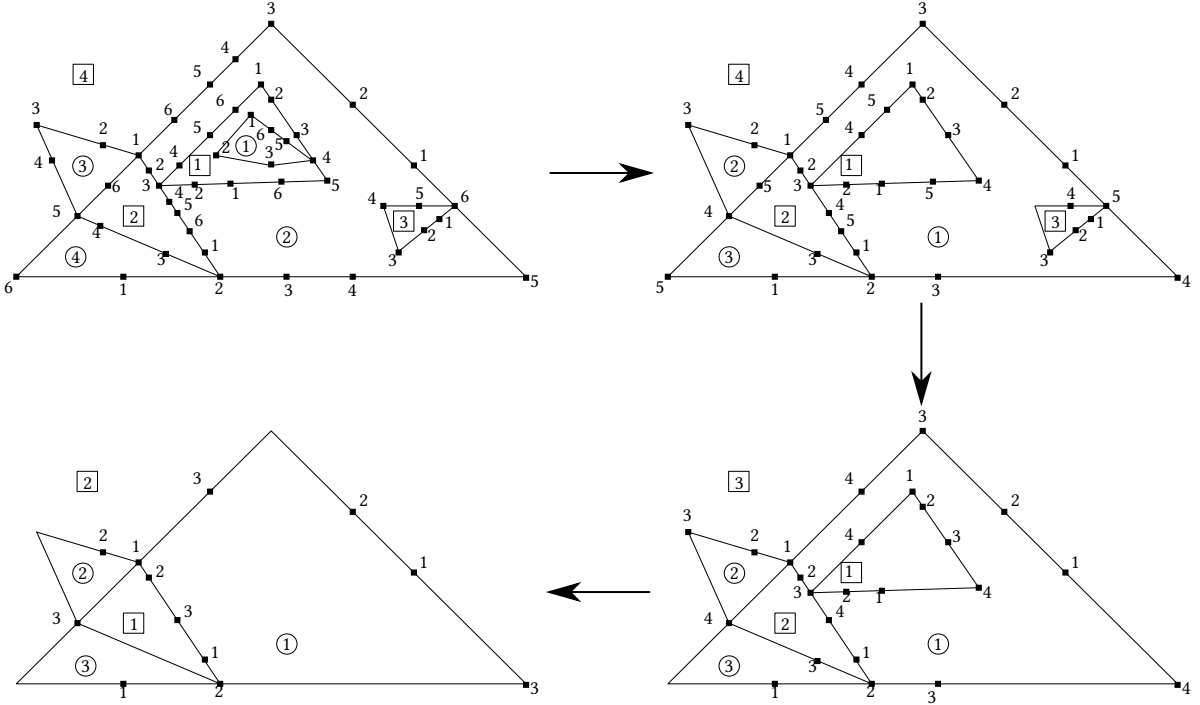
\caption{The application of \cref{con:prun} to a Hurwitz galaxy of type $(0,(1,6,1,1),(3,2,1,3))$. The white faces are indicated by circled face labellings and the black faces by boxed face labellings.}
\label{fig:ex}
\end{figure}
\end{example}

One of the important subtleties concerning the pruning process is that the order in which the loop faces is removed is not unique while the result of the process is. We now define relations between tuples of partitions reflecting the bi-pruning process.

\begin{definition}
Let $\mu,\mu'$ be two ordered partitions and $I\subset[\ell(\mu)]$. We write $\mu'\preceq_I\mu$ if $\ell(\mu')\le\ell(\mu)$, $|I|=\ell(\mu')$ and $\mu'\le\mu_I$ entrywise.\par
Moreover, for partitions $\mu',\nu',\mu,\nu$ we write $(\mu',\nu')\preceq_{I,J}(\mu,\nu)$, if $\mu'\preceq_I\mu,\nu'\preceq_J\nu$ for $I\subset[\ell(\mu)],J\subset[\ell(\nu)]$ and $|\mu'|=|\nu'|$.
\end{definition}

The goal of this section is formulate \cref{thm:main1}, which essentially states that double Hurwitz numbers are determined by bi-pruned double Hurwitz numbers. More precisely, we state a formula expressing double Hurwitz numbers as a weighted sum over bi-pruned double Hurwitz numbers with smaller input data. We now explain the structure of our proof of \ref{thm:main1} in \cref{sec:proof}, in order to motivate the definitions needed in order to state the theorem.
\begin{enumerate}
\item We observe that the statement in \cref{thm:main1} is a weighted bijection between $h_g(\mu,\nu)$ and $Ph_g(\mu',\nu')$ for all $(\mu',\nu')\preceq_{I,J}(\mu,\nu)$.
\item The idea of the proof is to show that there is a weighted bijection between all Hurwitz galaxies of type $(g,\mu,\nu)$ and bi-pruned Hurwitz galaxies of type $(g,\mu',\nu')$ for $(\mu',\nu')\preceq_{I,J}(\mu,\nu)$ as before, i.e. we map each Hurwitz galaxy $G$ of type $(g,\mu,\nu)$ to a tuple $(G',I,J)$, where $G'$ is a unique bi-pruned Hurwitz galaxy of type $(g,\mu',\nu')$ for $(\mu',\nu')\preceq_{I,J}(\mu,\nu)$ and the cardinality of the pre-image of each bi-pruned Hurwitz galaxy only depends on the data $(\mu',\nu')\preceq_{I,J}(\mu,\nu)$.
\item The map is given by the bi-pruning process, i.e. to each Hurwitz galaxy $G$, we associate its underlying bi-pruned Hurwitz galaxy $G'$, where $I$ and $J$ capture the labels of the faces of $G'$ in $G$.
\item For fixed $g,\mu,\nu$, we let $\mu',\nu',I,J$, such that $(\mu',\nu')\preceq_{I,J}(\mu,\nu)$ and we fix a bi-pruned Hurwitz galaxy $G'$ of type $(g,\mu',\nu')$. We now analyse the number Hurwitz galaxies of type $(g,\mu,\nu)$ whose underlying bi-pruned Hurwitz galaxy is $G'$.
\item This analysis is done by reversing the bi-pruning process, i.e. we consecutively glue white and black loop faces into $G'$, such that we obtain a Hurwitz galaxy of type $(g,\mu,\nu)$.
\item The key observation in \cref{sec:proof} is that we can focus on the combinatorics on the level of partitions: Let $G'_1$ be a Hurwitz galaxy obtained by gluing a loop face into $G'$. Then $G'_1$ is of type $(g,\tilde{\mu}^1,\tilde{\nu}^1)$. The data $(\tilde{\mu}^1,\tilde{\nu}^1)$ is of course related to the data $\mu,\nu,\mu',\nu',I,J$. This relation is conceptionalised in \cref{def:gluestep} and we say $(\tilde{\mu}^1,\tilde{\nu}^1)$ is obtained from $(\mu',\nu')$ by a \textit{gluing step}. When we distguish the colour of the loop face, we glued into $G'$, we speak of \textit{black gluing step} (for a black loop face) or a \textit{white gluing step} (for a white loop face).
\item Thus by reversing the bi-pruning process, we obtain a sequence of gluing steps, which start from $(\mu',\nu')$ and result in $(\mu,\nu)$. We refer to the conceptionalisation of this idea as a \textit{gluing sequence from $(\mu',\nu')$ to $(\mu,\nu)$ for $I,J$} in \cref{def:seq}.
\item In \cref{sec:proof}, we observe that each gluing sequence $S$ from $(\mu',\nu')$ to $\mu,\nu$  for $I,J$ governs several ways of consecutively gluing loop faces into $G'$.
\item This number does not depend on $G'$ but only on $(g,\mu',\nu'),I,J,S$. Moreover, it can be explicitly stated by observing that the ways of gluing loop faces into $G'$ according to $S$ can be expressed in terms of families of forests with fixed numbers of components. This is what the multiplicities in \cref{def:mult} correspond to. 
\end{enumerate}

We now introduce the technical notions touched upon in the previous outline of our proof. Once again, we stress that while technical all definitions correspond to explicit notions related to consecutive gluing processes of loop faces, which appear in the proof of \cref{thm:main1}.

\begin{definition}
\label{def:gluestep}
Let $\mu,\nu,\mu',\nu'$ be partitions such that $(\mu',\nu')\preceq_{I,J}(\mu,\nu)$ for some $I\subset[\ell(\mu)],J\subset[\ell(\nu)]$. We fix $i_1\notin I$, $i_2\in I$, $s\le\mu_{i_1}$, $j_1\notin J$, $j_2\in J$, $t\le\nu_{j_1}$.\par 
We first define two new partitions $\tilde{\mu},\tilde{\nu}$ indexed by $I,J\cup\{j_1\}$ respectively, which correspond to gluing a black face of perimeter $t$ labelled $j_1$ into a white face labeled $i_2$:
\begin{itemize}
\item $\tilde{\mu}_k=\mu'_k$ for $k\neq i_2$,
\item $\tilde{\mu}_{i_2}=\mu'_{i_2}+t$,
\item $\tilde{\nu}_k=\nu'_k$ for $k\neq j_1$
\item $\tilde{\nu}_{j_1}=t$.
\end{itemize}
We say $(\tilde{\mu},\tilde{\nu})$ is obtained from $(\mu',\nu')$ by a  \textit{black gluing step of type} $(\mu',\nu',\mu,\nu,I,J,i_2,j_1,t)^{\bullet}$.\par
Similarly, we define two new partitions $\tilde{\mu},\tilde{\nu}$ indexed by $I\cup \{i_1\},J$ respectively, which correspond to gluing a white face of perimter $s$ labelled $j_2$ into a white face labeled $i_1$:
\begin{itemize}
\item $\tilde{\mu}_k=\mu'_k$ for $k\neq i_1$,
\item $\tilde{\mu}_{i_1}=s$,
\item $\tilde{\nu}_k=\nu'_k$ for $k\neq j_2$
\item $\tilde{\nu}_{j_2}=\nu'_{j_2}+s$.
\end{itemize}
We say $(\tilde{\mu},\tilde{\nu})$ is obtained from $(\mu',\nu')$ by a  \textit{white gluing step of type} $(\mu',\nu',\mu,\nu,I,J,i_1,j_2,s)^{\circ}$.
\end{definition}

We now consider sequences of gluings between two tuples of partitions $(\mu',\nu')$ and $(\mu,\nu)$.

\begin{definition}
\label{def:seq}
Let $\mu,\nu,\mu',\nu'$ be partitions, such that $\mu'\preceq_I\mu,\nu'\preceq_J\nu$ for some $I\subset[\ell(\mu)],J\subset[\ell(\nu)]$. We define a gluing sequence from $(\mu',\nu')$ to $(\mu,\nu)$ to be a  sequence of tuples of partitions $((\mu^0,\nu^0),\dots,(\mu^d,\nu^d))$ and gluing steps $(\tau_1,\dots,\tau_d)$, where $\tau_i=(\mu^i,\nu^i,\mu,\nu,I_i,J_i,k_i,l_i,s_i)^{\circ}$ or $\tau_i=(\mu^i,\nu^i,\mu,\nu,I_i,J_i,k_i,l_i,s_i)^{\bullet}$, such that
\begin{enumerate}
\item $\mu^0=\mu',\nu^0=\nu'$ and $\mu^d=\mu,\nu^d=\nu$,
\item $(\mu^{i+1},\nu^{i+1})$ is obtained from $(\mu^i,\nu^{i})$ by $\tau_{i+1}$,
\item $I_1=I,J_1=J$ and $I_d=[\ell(\mu)],J_d=[\ell(\nu)]$,
\item $I_{i+1}=I_i\cup\{k_i\}$ and $J_{i+1}=J_i\cup\{l_i\}$.
\end{enumerate}
For a gluing sequence $S$, we define its combinatorial type to be the sets $C_S^{\circ}\coloneqq\{(k_i,l_i,s_i)^{\circ}\}$, $C_S^{\bullet}\coloneqq\{(k_i,l_i,s_i)^{\bullet}\}$. We call two gluing sequences $S$ and $S'$ equivalent if $C_{S}^{\circ}=C_{S'}^{\circ}$ and $C_S^{\bullet}=C_{S'}^{\bullet}$. Moreover, we denote the set of equivalence classes of gluing sequences from $(\mu',\nu')$ and $(\mu,\nu)$ with respect to $I,J$ as above by $\mathcal{S}_{I,J}((\mu',\nu'),(\mu,\nu))$.
\end{definition}

\begin{remark}
While \cref{def:seq} may seem technical at first, all notions have a direct meaning in terms of gluing loop faces consecutively into Hurwitz galaxies. This becomes clearer in the discussion in\cref{proof1}. Here we give a rough idea: We start with a Hurwitz galaxy $G$ of type $(g,\mu',\nu')$ and consecutively glue loop faces into $G$ to obtain a Hurwitz galaxy of type $(g,\mu,\nu)$. A gluing sequence encodes certain data about this process:
\begin{itemize}
\item Each white (black) gluing step corresponds to gluing a white (resp. black) loop face into a black (resp. white) face.
\item The conditions (1)--(4) correspond to the fact, that each gluing step $\tau_i$ produces a new Hurwitz galaxy $G_i$ with an additional loop face. The gluing step $\tau_{i+1}$ should then correspond to gluing a loop face into $G_i$. Thus, we see that (1)--(4) are compatibility conditions.
\item The sets $C_S^{\circ}\coloneqq\{(k_i,l_i,s_i)^{\circ}\}$ and $C_S^{\bullet}\coloneqq\{(k_i,l_i,s_i)^{\bullet}\}$ collect the data, which white (resp. black) loop faces are glued into which black (resp.) white faces, i.e. they remember the labels of the faces. Moreover, the integer $s_i$ collects the data of the perimeter of the loop face, which is glued into the Hurwitz galaxy.
%\item The notion of $\mu_i^S$ and $\nu_j^S$ corresponds to the following subtlety: After a loop face $F$ is glued into the Hurwitz galaxy, other loop faces may be glued into $F$, which changes the perimeter of $F$. The integers $\mu_i^S$ and $\nu_j^S$ collect the initial perimeter of the white face labeled $i$ and the black face labeled $j$. If $i$ and $j$ are contained in $G$, i.e. the Hurwitz galaxy we started with, then $\mu_i^S$ and $\nu_j^S$ collect the perimeter of the white face labeled $i$ and the black face labeled $j$ in $G$.
\end{itemize} 
\end{remark}

We now associate a multiplicity to any given equivalence class of gluing sequences.

\begin{definition}
\label{def:mult}
Let $\mu,\nu,\mu',\nu'$ be partitions such that $\mu'\preceq_I\mu$ and $\nu'\preceq_J\mu$ for some $I\subset[\ell(\mu)]$ and $J\subset[\ell(\nu)]$. Further, let $S\in\mathcal{S}_{I,J}((\mu',\nu'),(\mu,\nu))$. For $i\in [\ell(\mu)]$ ($j\in[\ell(\nu)]$) we denote by $I^S_i$ (resp. $J^S_j$) the subset of indices in $p\in[d]$, such that $\tau_p$ is black gluing step (resp. white gluing step) and $k_p=i$ (resp. $l_p=j$).

We define multiplicities for all $i\in[\ell(\mu)]$ by

\begin{align}
\mathrm{mult}^{\bullet}_{S;I,J}(i)=\sum_{\substack{\underline{a}\in\mathbb{Z}^{\mu_i^S+|I^S_i|}:\\
|\underline{a}|=|I^S_i|}}\sum_{k\in [\mu_i^S]}\binom{|I^S_i|-1}{a_1,\dots,a_{k-1},a_k-1,a_{k+1},\dots,a_{\mu_i^S+|I^S_i|}}\prod_{p\in I^S_i}s_p^{a_p}.
\end{align}

and the multiplicities for all $j\in[\ell(\nu)]$ by

\begin{align}
\mathrm{mult}^{\circ}_{S;i,J}(j)=\sum_{\substack{\underline{a}\in\mathbb{Z}^{\nu_j^S+|J^S_j|}:\\
|\underline{a}|=|J^S_j|}}\sum_{k\in [\nu_j^S]}\binom{|J^S_j|-1}{a_1,\dots,a_{k-1},a_k-1,a_{k+1},\dots,a_{\nu_j^S+|J^S_j|}}\prod_{p\in J^S_j}s_p^{a_p}.
\end{align}

Moreover, we define the multiplicity of the gluing sequence $S$ from $(\mu',\nu')$ to $(\mu,\nu)$ with respect to $I$ as
\begin{equation}
\mathrm{mult}_{I,J}(S)=\binom{2g-2+\ell(\mu)+\ell(\nu)}{2g-2+\ell(\mu')+\ell(\nu')}\cdot (|I^c|+|J^c|)!\cdot \prod_{i\in\ell(\mu)}\mathrm{mult}_{S}(i)\prod_{j\in\ell(\nu)}\mathrm{mult}_{S}(j).
\end{equation}
\end{definition}

\begin{theorem}
\label{thm:main1}
Let $\mu,\nu$ be partitions of the positive integer and let $g$ be a non-negative integer, such that $\ell(\mu)+\ell(\nu)\neq2$. Then, we have
\begin{align}
\label{equ:main}
h_g(\mu,\nu)=&\sum_{\substack{I\subset[\ell(\mu)]\\J\subset[\ell(\nu)]}}\sum_{\substack{\mu'\preceq_I\mu\\\nu'\preceq_J\nu}}\widehat{Ph}_g(\mu',\nu')\sum_{S\in\mathcal{S}_{I,J}((\mu',\nu'),(\mu,\nu))}\mathrm{mult}_{I,J}(S)\\
+&\left(\sum_{j\in[\ell(\nu)]}\sum_{a=1}^{\mathrm{min}(\mu_1,\nu_j)}\sum_{S\in\mathcal{S}_{\{1\},\{j\}}((a,a),(\mu,\nu))}\mathrm{mult}_{\{1\},\{j\}}(S)\right)\cdot\delta_{g,0},
\end{align}
where $\delta_{g,0}$ is the Kronecker symbol.
\end{theorem}

\begin{remark}
We note that the system of equations given in \cref{equ:main} is of lower-triangular shape with respect to $\mu$ and $\nu$. Thus the families of numbers $h_g(\mu,\nu)$ and $\widehat{Ph}_g(\mu,\nu)$ determine each other.
\end{remark}

The next example illustrates \cref{thm:main1}.

\begin{example}
\label{ex:comp2}
As in \ref{ex:comp}, we fix $g=0,\mu=(2,1),\nu=(1,1,1)$. We first observe all tuples of partitions, such that $(\mu',\nu')\preceq_{I,J}(\mu,\nu)$ for some $I\subset[2],J\subset[3]$:
\begin{enumerate}
\item $((1),(1))\preceq_{I,J}(\mu,\nu)\quad\mathrm{for}\quad I=\{1\},\{2\},J=\{1\},\{2\},\{3\}$,
\item $((2),(1,1))\preceq_{I,J}(\mu,\nu)\quad\mathrm{for}\quad I=\{1\},J=\{1,2\},\{2,3\},\{1,3\}$,
\item $((1,1),(1,1)\preceq_{I,J}(\mu,\nu)\quad\mathrm{for}\quad I=[2],J=\{1,2\},\{2,3\},\{1,3\}$,
\item $((2,1),(1,1,1)\preceq_{I,J}(\mu,\nu)\quad\mathrm{for}\quad I=[2],J=[3]$.
\end{enumerate}
We see that only for (3) and (4) the set $\mathcal{S}_{I,J}((\mu',\nu'),(\mu,\nu))$ is non-empty. For (4), the empty gluing sequence is the only one, as any bi-pruned Hurwitz galaxy of type $(0,(2,1),(1,1,1))$ is a also a Hurwitz galaxy of type $(0,(2,1),(1,1,1))$. Thus, the contribution in case (4), is the summand
\begin{equation}
Ph_0((2,1),(1,1,1))\cdot 1=6.
\end{equation}
Now, we consider case (3). We fix $I=[2]$ and $J=[2]$. We see that the only gluing sequence in $\mathcal{S}_{I,J}(((1,1),(1,1)),((2,1),(1,1,1)))$ is given by $S=(((1,1),(1,1),(2,1),(1,1,1),[2],[2],1,3,1))^{\bullet}$. Moreover, the multiplicity of each face of $1$ and we obtain
\begin{equation}
\mathrm{mult}_{I,J}(S)=\binom{3}{2}\cdot 1\cdot 1\cdot 1=3.
\end{equation}
We proceed similarly for $J=\{1,3\}$ and $J=\{2,3\}$ and obtain the same result for $\mathrm{mult}_{I,J}(S)$. This yields a summand
\begin{equation}
Ph_0((1,1),(1,1))\cdot(3+3+3)=2\cdot(3+3+3)=18.
\end{equation}
The term
\begin{equation}
\left(\sum_{j\in[\ell(\nu)]}\sum_{a=1}^{\mathrm{min}(\mu_1,\nu_j)}\sum_{S\in\mathcal{S}_{\{1\},\{j\}}((a,a),(\mu,\nu))}\mathrm{mult}_{\{1\},\{j\}}(S)\right)\cdot\delta_{g,0}
\end{equation}
yields a summand of $0$ in this case, as $\mathcal{S}_{\{1\},\{j\}}(a,a)$ is empty for any choice of $j$ and $a$.
This yields
\begin{equation}
h_0((2,1),(1,1,1))=Ph_0((2,1),(1,1,1))\cdot 1+Ph_0((1,1),(1,1))\cdot 9=6+18=24,
\end{equation}
which coincides with our computation in \cref{ex:comp}.
\end{example}

\section{Structural observations}
\label{sec:struc}
In this section, we make two observations about the structure of bi-pruned double Hurwitz numbers, which mirror the structure of double Hurwitz numbers.\par
We fix the configuration space of partitions of the same size with fixed lengths $W=\{(\mu,\nu)\in\mathbb{N}^m\times\mathbb{N}^n\mid\sum_{i=1}^m\mu_i=\sum_{j=1}^n\nu_j\}$. Then we can view double Hurwitz numbers as a map
\begin{align}
h_g:W&\to\mathbb{Q}\\
(\mu,\nu)&\mapsto h_g(\mu,\nu).
\end{align}

The family $(\sum_{i\in I}\mu_i=\sum_{j\in J}\nu_j)_{I\in[m],J\in[n]}$ yields a hyperplane arrangement inside $W$. We consider its complement and call the connected components \textit{chambers}. It was proved in \cite{GJVdouble} that for any chamber $C$, there exists a polynomial $p_{g;m,n}^C$ in $m+n$ variables of degree $4g-3+m+n$, such that
\begin{equation}
h_{g}(\mu,\nu)=p_{g;m,n}^C(\mu,\nu)
\end{equation}
for all $(\mu,\nu)\in C$. For pruned Hurwitz numbers, we obtain the following analogous result.

\begin{proposition}
\label{prop:poly}
Let $m,n$ be positive integers and $g$ a non-negative integer. Let $W=\{(\mu,\nu)\in\mathbb{N}^m\times\mathbb{N}^n\mid\sum_{i=1}^m\mu_i=\sum_{j=1}^n\nu_j\}$. We view bi-pruned double Hurwitz numbers as a function
\begin{align}
Ph_g:W&\to\mathbb{Q}\\
(\mu,\nu)&\mapsto Ph_g(\mu,\nu).
\end{align}
Let $C$ be a connected component of the complement of the hyperplane arrangement $(\sum_{i\in I}\mu_i=\sum_{j\in J}\nu_j)_{I\in[m],J\in[n]}$ in $W$. Then there exists a polynomial $P_{g;m,n}^C$ of degree at most $4g-3+m+n$ in $m+n$ variables, such that
\begin{equation}
Ph_g(\mu,\nu)=P_{g;m,n}^C(\mu,\nu)
\end{equation}
for all $(\mu,\nu)\in C$.
\end{proposition}

\begin{proof}
The proof is parallel to the proof of piecewise polynomiality of double Hurwitz numbers in \cite{GJVdouble}, which is in terms of branching graphs. We recall the structure of the proof in \cite{GJVdouble}: Consider all reduced branching graphs on $m$ vertices with $n$ faces and $b=2g-2+m+n$ edges for some $g$. The key observation is that for a fixed reduced branching graph $\Gamma$ with this data, the number $\omega(\Gamma)$ of branching graphs of type $(g,\mu,\nu)$ with $\ell(\mu)=m,\ell(\nu)=n$ is piecewise polynomial in the entries of $\mu$ and $\nu$ of degree at most $4g-3+m+n$.\par 
We have alread observed in \cref{prop:bibranch} that bi-pruned double Hurwitz numbers can be expressed as the sum over all branching graphs without leaves and faces that are bounded by a single edge which is a loop. However, this is in fact a condition on the underlying reduced branching graph. Thus, we can express bi-pruned double Hurwitz numbers as follows
\begin{equation}
Ph_g(\mu,\nu)=\sum \omega(\Gamma),
\end{equation}
where we sum over all reduced branching graphs without leaves and faces that are bounded by a single edge which is a loop. As this is a finite sum and since $\omega(\Gamma)$ behaves piecewise polynomially in $\mu$ and $\nu$ as proved in \cite{GJVdouble}, we obtain our result.
\end{proof}

We now illustrate the polynomial behaviour of bi-pruned double Hurwitz numbers.

\begin{example}
We fix $g=0,m=n=2$. All possible branching graphs for this case are illustrated in Figure 10 in \cite{Hahnpruned}. We see that the only pruned and bi-pruned branching graphs co-incide. Thus, we obtain for $d>a,b>c$
\begin{equation}
h_g((a,b),(c,d))=2d,\ Ph_g((a,b),(c,d))=2c.
\end{equation}
For the pruned double Hurwitz numbers, we obtain $\mathcal{PH}_g((a,b),(c,d))=2c$.\par 
We fix $g=0,m=1,n=2$. Thus, we obtain $b=2g-2+m+n=1$. The only possible reduced branching graph $\Gamma$ is a vertex with a edge, which is a loop, i.e. it is pruned but not bi-pruned. Moreover, for fixed $\mu=a$ and $\nu=(b,c)$, there is only one branching graph, with underlying reduced branching graph $\Gamma$. Thus, we obtain $h_g((a),(b,c))=\mathcal{PH}_g((a),(b,c))=1$ and $Ph_g((a),(b,c))=0$.
\end{example}

In the next step, we interpret bi-pruned double Hurwitz numbers in terms of factorisations in the symmetric group. The original result for the usual double Hurwitz numbers is essentially due to Hurwitz. In order to state it, we introduce the notion of \textit{factorisations of type} $(g,\mu,\nu)$. Let $(\sigma_1,\tau_1,\dots,\tau_b,\sigma_2)$ be a factorisation of type $(g,\mu,\nu)$ if
\begin{enumerate}
\item $\sigma_1,\ \sigma_2,\ \tau_i\in\mathcal{S}_d$, where $d=|\mu|=|\nu|$;
\item $\mathcal{C}(\sigma_1)=\mu_1,\ \mathcal{C}(\sigma_2)=\mu_2\textrm{ and }\mathcal{C}(\tau_i)=(2,1,\dots,1)$;
\item the group generated by $(\sigma_1,\tau_1,\dots,\tau_b,\sigma_2)$ acts transitively on $\{1,\dots,d\}$;
\item the disjoint cycles of $\sigma_1$ and $\sigma_2$ are labelled, such that the cycle of $\sigma_1$ ($\sigma_2)$ labelled $i$ (resp. $j$) is of length $\mu_i$ (resp. $\nu_j$);
\item $\tau_b\cdots\tau_1\cdot\sigma_1=\sigma_2$.
\end{enumerate}
We further denote by $F(g,\mu,\nu)$ the collection of factorisations of type $(g,\mu,\nu)$. Moreover, we call $(\sigma_1,\tau_1,\dots,\tau_b,\sigma_2)$ an \textit{inverted factorisation of type} $(g,\mu,\nu)$ if it satisfies the above conditions (1)--(4) and
\begin{itemize}
\item $\sigma_1\tau_1\cdots\tau_b=\sigma_2$.
\end{itemize}
We denote by $F^{in}(g,\mu,\nu)$ the collection of inverted factorisations of type $(g,\mu,\nu)$.\par 
One can pass from a factorisation of type $(g,\mu,\nu)$ to an inverted factorisation of type $(g,\mu,\nu)$ and back as follows: Let $(\sigma_1,\tau_1,\dots,\tau_b,\sigma_2)$ be a factorisation of type $(g,\mu,\nu)$, we define $\eta_0=\sigma_1$, $\eta_i=\tau_i\cdots\tau_1\sigma_1$ and $\pi_i=\eta_{i-1}^{-1}\eta_i$ for $i=1,\dots,b$. Then $(\sigma_1,\pi_1,\dots,\pi_b,\sigma_2)$ is an inverted factorisation of type $(g,\mu,\nu)$. Similarly, we obtain $\eta_i=\sigma_1\pi_1\cdots\pi_i$ and $\tau_i=\eta_i\eta_{i-1}^{-1}$ for $i=1,\dots,b$.\par 
The following theorem is essentially due to Hurwitz.

\begin{theorem}
Let $g$ be a non-negative integer and $\mu,\nu$ partitions of the same positive integer. Then, we have
\begin{equation}
h_g(\mu,\nu)=\frac{1}{d!}|F(g,\mu,\nu)|=\frac{1}{d!}|F^{in}(g,\mu,\nu)|.
\end{equation}
\end{theorem}

In addition to it being a fascinating observation, this theorem enables a feasible method of computing double Hurwitz numbers. Our next goal is to derive an analogous result for the bi-pruned case.\par 
We begin by discussing an algorithm, which associates tuples of factorisations to a Hurwitz galaxy due to Johnson. We note, that instead of extracting the factorisation $(\sigma_1,\tau_1,\dots,\tau_b,\sigma_2)$ from the Hurwitz galaxy, we extract $\sigma_1,\tau_1\sigma_1,\tau_2\tau_1\sigma_1,\dots,\tau_{b-1}\cdots\tau_1\sigma_1,\tau_b\cdots\tau_1\sigma_1(=\sigma_2)$.

\begin{algorithm}[\cite{Johnsontropicalization}]
\label{alg:sym}
Let $g$ be a non-negative integer, $\mu$ and $\nu$ partitions of the same positive integer $d$ and $G$ a Hurwitz galaxy of type $(g,\mu,\nu)$.
\begin{enumerate}
\item There are $d$ many edges connecting a vertex labeled $1$ and a vertex labeled $b$. We assign a marking in $[d]$ to each such edge, such that all edges are marked distinctively. The arbitrary character of the marking corresponds to the fact that the permutations are only determined up to conjugation.
\item To obtain $\eta_0$, we fix a marking $x$, which is adjacent to white face. We proceed counterclockwisely with respect to this white face from $x$ along the border of the white until we reach the next marking $y$. Then $x$ is mapped to $y$. By proceeding accordingly for all marking, we obtain a permutation, which we call $\eta$.
\item In order to obtain $\eta_i$ for $i>0$, we one again fix a marking $x$ and proceed along the same orientation as in step (2). If we reach a marking $y$ before we reach a vertex, then $x$ is mapped to $y$. However, if we reach a vertex before, we reach a marking, we proceed differentl: In step (2), we turned left at each vertex due to the orientation of the border with respect to the white faces. To obtain $\eta_i$, we turn right all vertices labeled $j$ for $j\le i$ and obtain a permutation as before.
\end{enumerate}
This way we obtain permutations $\eta_0,\dots,\eta_b$. We obtain a factorisation $(\sigma_1,\tau_1,\dots,\tau_b,\sigma_2)$ of type $(g,\mu,\nu)$ by 
\begin{itemize}
\item $\sigma_1=\eta_0$,
\item $\tau_i\cdots\tau_1\sigma_1=\eta_i$, i.e. $\tau_i=\eta_i\sigma_1^{-1}\tau_1^{-1}\cdots\tau_{i-1}^{-1}=\eta_i\eta_{i-1}^{-1}$ for $i=1,\dots,b$,
\item $\sigma_2=\tau_b\cdots\tau_1\sigma_1$.
\end{itemize}
Similarely, we obtain an inverted factorisation $(\sigma_1,\pi_1,\dots,\pi_b,\sigma_2)$ of type $(g,\mu,\nu)$ by
\begin{itemize}
\item $\sigma_1=\eta_0$,
\item $\sigma_1\pi_1\cdots\pi_i=\eta_i$, i.e. $\pi_i=\pi_{i-1}^{-1}\cdots\pi_1^{-1}\sigma_1^{-1}\eta_i=\eta_{i-1}^{-1}\eta_i$ for $i=1,\dots,b$,
\item $\sigma_2=\sigma_1\tau_b\cdots\tau_1$.
\end{itemize}
\end{algorithm}

Before, we state the next lemma, we introduce some notation. For a permutation $\sigma\in S_d$, we denote by $\mathrm{supp}(\sigma)$ those elements of $[d]$, which are changed under the natural action of $\sigma$. Moreover, for a labeled permutation $\sigma$, we denote the cycle labeled $i$ by $\sigma^i$.

\begin{lemma}
\label{lem:sym}
Let $G$ be a Hurwitz galaxy with edge marking in $\{1,\dots,d\}$ as in \cref{alg:sym} of type $(g,\mu,\nu)$, let $(\sigma_1,\tau_1,\dots,\tau_b,\sigma_2)$ be the factorisation and $(\sigma_1,\pi_1,\dots,\pi_b,\sigma_2)$ the inverted factorisation associated to $G$.
\begin{enumerate}
\item $G$ contains a white loop face if and only if there exists an $i\in[\ell(\mu)]$, such that there is exactly one $j\in[b]$ with $\mathrm{supp}(\tau_j)\cap\mathrm{supp}(\sigma_1^i)\neq\emptyset$,
\item $G$ contains a black loop face if and only if there exists an $i\in[\ell(\nu)]$, such that there is exactly one $j\in[b]$ with $\mathrm{supp}(\pi_j)\cap\mathrm{supp}(\sigma_2^i)\neq\emptyset$.
\end{enumerate}
\end{lemma}

\begin{proof}
This is seen immediatly by extracting the transpositions $\tau_j$ and $\pi_k$ from the Hurwitz galaxies (see \cref{ex:sym} for an illustration of this procedure). In order to obtain $\tau_j$ from $G$, we consider the $4-$valent vertex labeled $j$. This vertex labeled $j$ is adjacent to two white face $i_1,i_2$ (not necessarily distinct). Moreover, the $4-$valent vertex labeled $j$ is incident to two edges $e_1$ and $e_2$ which are adjacent to a vertex labeled $j+1$. Of those edges $e_1,e_2$, one is adjacent to the white face labeled $i_1$ and one is adjacent to white face labeled $i_2$. By convention let $e_1$ be adjacent to $i_1$ and $e_2$ adjacent to $i_2$. Then we proceed from the $4-$valent vertex labeled $j$ counterclockwisely with respect to $i_1$ along $e_1$ until we reach a marking $x$. Similarely, we proceed from the $4-$valent vertex labeled $j$ counterlockwisely with respect to $i_2$ along $e_2$ until we reach a marking $y$. Then we obtain $\tau_j=(x\ y)$.\par 
In order to obtain $\pi_k$, we proceed similarely. We consider the $4-$valent vertex labeled $k$, which is adjacent to two black faces labeled $i_1,i_2$. Then the $4-$valent vertex labeled $k$ is incident to two edges $e_1,e_2$, which are adjacent to a vertex labeled $k-1$. As before, let $e_1$ be adjacent to the black face labeled $i_1$ and $e_2$ adjacent to the black face labeled $i_2$. We proceed from the $4-$valent vertex labeled $k$ counterclockwisely with respect to $i_1$ ($i_2$) along $e_1$ (resp. $e_2$) until we reach a marking $x$ (resp. $y$). Then, we obtain $\pi_k=(x\ y)$.\par 
Thus, if $\tau_i=(r_i\ s_i)$ ($\pi_i=(r_i\ s_i)$) corresponds to a vertex adjacent to the white faces (resp. black faces) $i_1,i_2$, then by construction, $r_i$ is contained in the support of the cycle corresponding to $i_1$ and $s_i$ is contained in the support of the cycle corresponding to $i_2$. This completes the proof of the lemma.
\end{proof}

This motivates the following definition.

\begin{definition}
Let $(\sigma_1,\tau_1,\dots,\tau_b,\sigma_2)$ be a factorisation of type $(g,\mu,\nu)$ and let $(\sigma_1,\pi_1,\dots,\pi_b,\sigma_2)$ be the associated inverted factorisation of type $(g,\mu,\nu)$. We call $(\sigma_1,\tau_1,\dots,\tau_b,\sigma_2)$ a \textit{bi-pruned factorisation of type} $(g,\mu,\nu)$ if
\begin{itemize}
\item for all $i\in[\ell(\mu)]$ there exist $j,k\in[b]$, $j\neq k$ with $\mathrm{supp}(\tau_j)\cap\mathrm{supp}(\sigma_1^i)\neq\emptyset\neq\mathrm{supp}(\tau_k)\cap\mathrm{supp}(\sigma_1^i)$ and
\item for all $i\in[\ell(\nu)]$ there exist $j,k\in[b]$, $j\neq k$ with $\mathrm{supp}(\pi_j)\cap\mathrm{supp}(\sigma_2^i)\neq\emptyset\neq\mathrm{supp}(\pi_k)\cap\mathrm{supp}(\sigma_2^i)$.
\end{itemize}
We denote by $F^{bi}(g,\mu,\nu)$ the set of all bi-pruned factorisations of type $(g,\mu,\nu)$.
\end{definition}

The above discussion is summarised in the following result.

\begin{proposition}
\label{prop:sym}
Let $\mu,\nu$ be two ordered partitions of the same positive integer and $g$ a non-negative integer. Then
\begin{equation}
Ph_g(\mu,\nu)=\frac{1}{d!}|F^{bi}(g,\mu,\nu)|.
\end{equation}
\qed
\end{proposition}

%
%\begin{proof}
%This follows immediately from the discussion in section 4 in \cite{Johnsontropicalization}. In this paper, we do not state the exact algorithm of obtaining a factorisation in the symmetric group from a Hurwitz galaxy. However, we illustrate the algorithm in \cref{ex:sym} in the case of loop faces.
%\end{proof}

We illustrate the procedure of associating permutations to a Hurwitz galaxy in the following example.

\begin{example}
\label{ex:sym}
\begin{enumerate}
\item We fix the Hurwitz galaxy of type $(0,(3,1),(2,1,1))$ in \cref{fig:exsym} and see $b=2\cdot0-2+\ell((3,1))+\ell((2,1,1))=-2+2+3=3$. We add extra markings $a_1,\dots,a_4$ between the vertices with label $1$ and label $3$ (see \cref{fig:exsym}). We now construct the permutation $\eta_0,\dots,\eta_3$. Proceeding as in \cref{alg:sym}, we obtain
\begin{equation}
\eta_0=(a_1\ a_2\ a_3)(a_4),\quad \eta_1=(a_1\ a_2\ a_3\ a_4),\quad \eta_2=(a_1\ a_3\ a_4)(a_2),\quad \eta_3=(a_1\ a_3)(a_2)(a_4).
\end{equation}
We first compute $\tau_i$ and $\pi_j$ directly from the $\eta_k$:
\begin{align}
&\tau_1=\eta_1\eta_0^{-1}=(a_1\ a_2\ a_3\ a_4)(a_1\ a_3\ a_2)(a_4)=(a_1\  a_4)\\
&\tau_2=\eta_2\eta_1^{-1}=(a_2\ a_3),\ \tau_3=\eta_3\eta_2^{-1}=(a_1\ a_4).
\end{align}
and
\begin{align}
&\pi_1=\eta_0^{-1}\eta_1=(a_1\ a_3\ a_2)(a_1\ a_2\ a_3\ a_4)=(a_3\ a_4)\\
&\pi_2=\eta_1^{-1}\eta_2=(a_1\ a_2),\ \pi_3=\eta_2^{-1}\eta_3=(a_3\ a_4).
\end{align}
We see that this is also reflected in the Hurwitz galaxy, when we the proof of \cref{lem:sym}, where a procedure to extract transpositions from Hurwitz galaxies is derived.\par 
In order to see that the black face labeled $3$ is a loop face we use the result \cref{lem:sym}. Namely, the black face labeled $3$ corresponds to the cycle $(a_2)$, which has support $\{a_2\}$. We further see $\mathrm{supp}(\pi_1)=\{a_3,a_4\}$, $\mathrm{supp}(\pi_2)=\{a_1,a_2\}$, $\mathrm{supp}(\pi_3)=\{a_3,a_4\}$. Thus, the only transpositions which non-empty intersection of the supports is $\pi_2$, i.e. $\mathrm{supp}(\pi_2)\cap\mathrm{supp}((a_2))=\{a_2\}\neq\emptyset$. Similarely, for all other faces, we can find \textbf{two} transpositions which yield non-trivial intersection of the supports (where we consider $\tau_i$ for the white faces and $\pi_j$ for the black faces).
\item We now consider the Hurwitz galaxy of type $(0,(2,1),(2,1))$ in \cref{fig:exsym2} and proceed as before. This yields
\begin{equation}
\eta_0=(a_1\ a_2),\quad \eta_1=(a_1\ a_3\ a_2),\quad \eta_2=(a_2\ a_3).
\end{equation}
This yields
\begin{equation}
\tau_1=(a_2\ a_3),\quad \tau_2=(a_1\ a_3)
\end{equation}
and
\begin{equation}
\pi_1=(a_1\ a_3),\quad \pi_2=(a_1\ a_2).
\end{equation}
By the same arguments as in the first example, this reflects the fact that the Hurwitz galaxy is bi-pruned.
\end{enumerate}
\end{example}

\begin{figure}
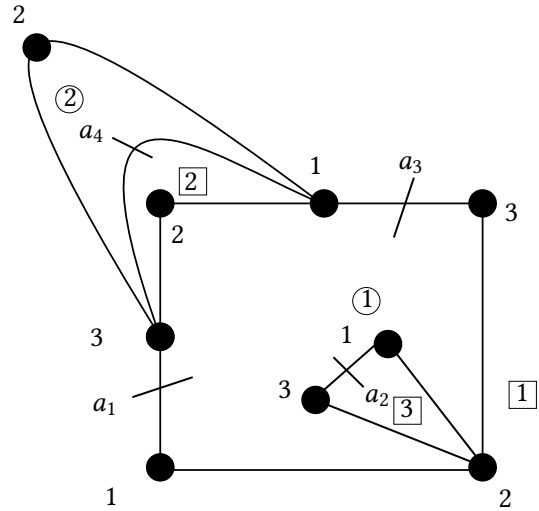
\caption{A Hurwitz galaxy with extra markings between vetices with label $1$ and $3$.}
\label{fig:exsym}
\end{figure}

\begin{figure}
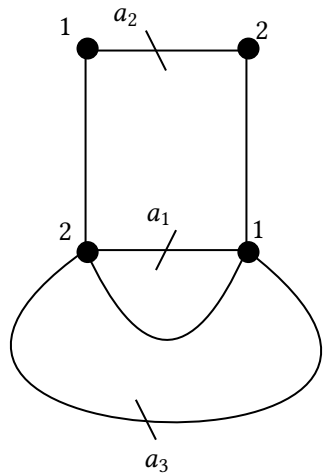
\caption{A Hurwitz galaxy with extra markings between vetices with label $1$ and $2$.}
\label{fig:exsym2}
\end{figure}

We finish this section with a remark on topological recursion.

\begin{remark}
We note that a discussion concerning pruned double Hurwitz numbers in the context of topological recursion appears in \cite{DKpruned}. It would be interesting to understand possible connections between topological recursion and bi-pruned double Hurwitz numbers, as well.
\end{remark}

\section{Proof of \cref{thm:main1}}
\label{sec:proof}
The structure of the proof is as follows: We begin by introducing the idea of associating a family of forests, i.e. (possibly disconnected graphs without cycles) to a  pruned Hurwitz galaxy and a gluing sequence. We then introduce the idea of gluing loop faces into Hurwitz galaxies in terms of data encoded in families of forests (\cref{proof1}). We finish the proof in \cref{proof2} by associating a family of forests and a gluing sequence to the pruning process of a fixed Hurwitz galaxy and relating it to the discussion in \cref{proof1}.\par 

For this analysis, let us define the notion of gluing loop faces into a Hurwitz galaxy $\Gamma$ of type $(g,\mu,\nu)$. Let $F$ be a face of $\Gamma$ of perimeter $b$. Thus, the boundary of $F$ is homeomorphic to a circle with $b\cdot m$ vertices, where $m=2g-2+\ell(\mu)+\ell(\nu)$. We now want to glue a black (white) loop face labelled $j$ for some $j\in[\ell(\nu)+1]$ (resp. $j\in[\ell(\mu)+1]$) of perimeter $a$ into $F$ of white (black) face labelled $i$ for $i\in[\ell(\mu)]$ (resp. $i\in[\ell(\nu)]$).

\begin{construction}
\label{con:glue}
For the above setting, we do the following construction.
\begin{enumerate}
\item We choose an integer $k\in[m+1]$.
\item We enrich $\Gamma$ with vertices by changing the labelling of all vertices with label $k'>k$ to $k'+1$ and adding a vertex labelled $k$ between all vertices labelled $k-1$ and $k+1$. This way obtain a circle with $b\cdot (m+1)$ vertices as the boundary of $F$.
\item We let $F'$ be a circle with $a\cdot(m+1)$ vertices. We choose one vertex labelled $k$ in $F'$ and one vertex labelled $k$ in $F$ and glue them together.
\item We further embed $F'$ into the face $F$ with the opposite orientation on the boundary, i.e. if $F$ is a white with vertices labelled cyclically counterclockwise, $F'$ is black face with vertices labelled cyclically clockwise and vice versa.
\item If $F$ is a white (black) face, we change labels of black (resp. white) faces, such that for all black (resp. white) faces of $\Gamma$ with label $l\ge j$, we increase the label by $1$ and label the black (resp. white) face $F'$ by $j$.
\end{enumerate}
We obtain a new Hurwitz galaxy $\Gamma'$ of type $(g,\tilde{\mu},\tilde{\nu})$, where if $F'$ is a black (white) face, the tuple $(\tilde{\mu},\tilde{\nu})$ is obtained from $(\mu,\nu)$ by the black (resp. white) gluing step $(\mu,\nu,\tilde{\mu},\tilde{\nu},[\ell(\mu)+1],[\ell(\nu)+1]\backslash\{j\},i,j,a)^{\bullet}$ (resp. $(\mu,\nu,\tilde{\mu},\tilde{\nu},[\ell(\mu)+1]\backslash\{j\},[\ell(\nu)+1],i,j,a)^{\circ}$).
\end{construction}

\subsection{Gluing sequences, faces and forests}
\label{proof1}
Let $I\subset[\ell(\mu)],J\subset[\ell(\nu)]$, $\mu'\preceq_I\mu,\nu'\preceq_J\nu$, let $\Gamma$ be a bi-pruned Hurwitz galaxy of type $(g,\mu',\nu')$ and let $S\in\mathcal{S}_{I,J}((\mu',\nu'),(\mu,\nu))$. Our first step is to contruct a Hurwitz galaxy of type $(g,\mu,\nu)$ from $\Gamma$ and $S$. We note, that in a sense this is a reversal of the pruning process.

\begin{construction}
\label{con:glue}
For the above setting for $\Gamma,I,J,\mu,\nu,\mu',\nu',g$, we introduce the following construction
\begin{enumerate}
\item We set $a=0$.
\item We set $M_a=I$, $M_a'=\emptyset$ and $N_a=J$, $N_a'=\emptyset$.
\item For each $i\in M_a\backslash M_a'$, we glue a loop face into the face labelled $i$ with respect to the gluing step $\tau_j$ for each $j\in I^S_i$. Here we have many choices, especially the order.
\item Let $j\in N_a\backslash N_a'$, for each $i\in J^S_j$, we glue in loop faces into $j$ consecutively according to the gluing step $\tau_i$. Here we have many choices, especially the order.
\item We obtain a new Hurwitz galaxy $\Gamma_a$ of type $(g,\tilde{\mu}_a,\tilde{\nu_a})$, where $\tilde{\mu}_a$ ($\tilde{\nu}_a$) is indexed by $I\cup\bigcup_{j\in N_a} J^S_j$ (resp. $J\cup\bigcup_{j\in M_a} I^S_i$).
\item We set $M_{a+1}'=M_a$, $M_{a+1}=I\cup\bigcup_{j\in N_a} J_S^j$, $N_{a+1}'=N_a$ and $N_{a+1}=J\cup\bigcup_{i\in M_{a}} I_S^i$. Moreover, we increase $a$ by $1$, i.e. $a\leftarrow a+1$ and repeat steps (3)-(6) until $M_a'=[\ell(\mu)]$ and $N_a'=[\ell(\nu)]$.
\item By construction, we obtain a Hurwitz galaxy $G$ of type $(g,\mu,\nu)$.
\end{enumerate}
\end{construction}

For the rest of the discussion it will be important to  capture the perimeter of each face, the first time it appears in a Hurwitz galaxy in the above reversal pruning process in \cref{con:glue} (resp. the perimeter of each face, the last time it appears in the pruning process in the Hurwitz galaxy). This can be also be phrased as follows: Starting from the Hurwitz galaxy $G$, we obtain in \cref{con:glue}, we fix a white face $W_i$ (resp. black face $B_j$) labeled $i$ (resp. $j$). We then start the pruning process until $W_i$ (resp. $B_j$) is a loop face. We then denote the perimeter of $W_i$ (resp. $B_j$) by $\mu_i^S$ (resp. $\nu_j^S$), where $S$ is the gluing sequence, we started with.\par 
It turns out that this numbers $\mu_i^S$, $\nu_j^S$ do not depend on $\Gamma$ or $G$ but only the gluing sequence $S$: For $i\notin I$ ($j\notin J$) and a gluing sequence $S$ as above, let $x_i$ (resp. $y_j$) be the first index, such that $l_{x_i}=i$ (resp. $k_{y_j}=j$). We then see that $\mu_i^{S}=s_{k_{x_i}}$ and $\nu_j^{S}=s_{l_{y_j}}$ for $i\notin I$ and $j\notin J$. This is due to the fact that the step $x_i$ (resp. $y_j$) is the step which glues the face $W_i$ (resp. $B_j$) into the Hurwitz galaxy for the first time.\par 
For $i\in I$ and $j\in J$, we see $\mu_i^{S}=\mu'_i$ and $\nu_j^S=\nu'_j$. This corresponds to the fact that first time the faces labeled $i$ and $j$ appear is in the initial bi-pruned Hurwitz galaxy.\par

A mentioned in step (3) and (4) in \cref{con:glue} there are many free choices in this construction. We now associate families of forests to $S$, which will give us a way to group many of those choices together.

\begin{definition}
A \textit{feasible family of forests associated to} $S$ is a set $\{(F_i)_{i\in[\ell(\mu)]},(F^j)_{j\in[\ell(\nu)]})\}$, where
\begin{itemize}
\item $F_i$ is a forest on the vertex set $V(F_i))=[\mu_i^S]\sqcup I_i^S$. It has $\mu_i^S$ many components and each element of $[\mu_i^S]$ is contained in a different component. Moreover, the edges of $F_i$ are labelled by $E_i\subset[2g-2+\ell(\mu)+\ell(\nu)]$.
\item $F^j$ is a forest on the vertex set $V(F^j))=[\nu_j^S]\sqcup J_j^S$. It has $\nu_j^S$ many components and each element of $[\nu_j^S]$ is contained in a different component. Moreover, the edges of $F^j$ are labelled by $E^j\subset[2g-2+\ell(\mu)+\ell(\nu)]$.
\item the elements of $\{(E_i)_{i\in[\ell(\mu)]},(E^j)_{j\in[\ell(\nu)]}\}$ are pairwise disjoint sets.
\end{itemize}
\end{definition}

This corresponds to the above notion of contructing a Hurwitz galaxy of type $(g,\mu,\nu)$ from $\Gamma$ and $S$ in the following sense: We fix a feasible family of forests $\{(F_i)_{i\in[\ell(\mu)]},(F^j)_{j\in[\ell(\nu)]})\}$ associated to $S$. We focus on the white faces, as the procedure for black faces is completely analogously. Let $i\in[\ell(\mu)]$, such that $i\in M_a\backslash M_{a-1}$. Thus, the Hurwitz galaxy $\Gamma_{a-1}$ is the first one in the construction, which contains a white face labelled $i$. By construction the perimeter of this face is $\mu_i^S$ in $\Gamma_a$, i.e. its boundary is homeomorphic to a cycle with $m_a\cdot\mu_i^S$ many vertices, where $m_a=2g-2+\ell(\tilde{\mu}_a)+\ell(\tilde{\nu}_a)$. We divide this cycle into $\mu_i^S$ many segment consisting of paths from $1$ to $m_a$ (see \cref{fig:seg} for the case of $m_a=\mu_i^S=4$).

\begin{figure}
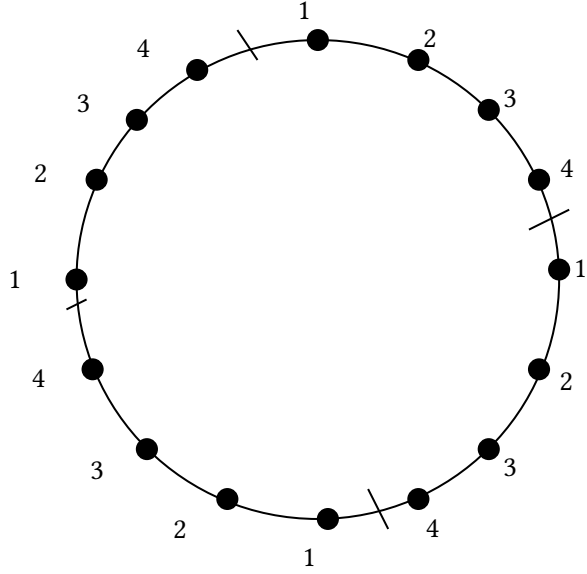
\caption{Dividing a face into segments.}
\label{fig:seg}
\end{figure}

Now, we glue faces into $i$ according to $\{(F_i)_{i\in[\ell(\mu)]},(F^j)_{j\in[\ell(\nu)]})\}$. In particular, we pick the forest $F_i$. This forest is labelled by $[\mu_i^S]\sqcup I_i^S$. Let $B_i^S\coloneq\{l_p\mid \tau_p\in I_i^S\}$. Then we can assume that $F_i$ is labelled by $[\mu_i^S]\sqcup B_i^S$, by re-labelling the vertex $\tau_p$ by $l_p$.
 
\begin{construction}
\label{con:forglue}
For the above setting, we glue all faces in $B_i^S$ into $i$ according to $F_i$.
\begin{enumerate}
\item We pick one connected component $C$ of $F_i$, which contains $r\in[\mu_i^S]$.
\item Let $k\in B_i^S$ be adjacent to $r$ with an edge labelled $e$, then we glue a loop face labelled $k$ of perimeter $\nu_k^S$ into the $r-$th segment (in the face labelled $i$) by attaching it to a vertex labelled $e$ as in \cref{con:glue}.
\item We proceed for all $k\in B_i^S$ adjacent to $r$ as in step (2).
\item Fix one $k$ as in step (2). For each $k'\in B_i^S$ adjacent to $k$ via an edge labelled $e'$, we glue a loop face into the face labelled by attaching it the face labelled $k$ at a vertex labelled $e'$.
\item We proceed as in step (4) for all adjacent vertices with labelled in $B_i^S$ in $C$.
\item We proceed as in steps (2)-(5) for all connected components $C$ of $F_i$.
\end{enumerate}
We note that the gluing for each $l_p\in B_i^S$ corresponds to the data $(k_p,l_p,s_p)^{\bullet}$ ($k_p=i)$), since in particular, $s_p=\nu_{l_p}^S$ by definition.
\end{construction}

We prove the following lemma.

\begin{lemma}
\label{lem:enum}
For the above setting let $a_p+1$ be the valency of the vertex labelled $\tau_p$. Then we have
\begin{equation}
\label{equ:multface}
\prod_{l_p\in B_i^S} (\nu_{l_p}^S)^{a_p}
\end{equation}
 choices of attaching loop faces in \cref{con:forglue}.
\end{lemma}

\begin{proof}
Then only choices we have are given by the fact, that we have not specified the vertices at which we attach two loop faces. We fix $k\in B_i^S$ and choose $k'\in B_i^S$, which is adjacent to $k$ via an edge labelled $e$, such that $k'$ has already been glued into the face labelled $i$ in \cref{con:forglue}, but $k$ has not. For $k$ ($k'$) , there are $\nu_k^S$ (resp. $\nu_{k'}^S$) many vertices labelled $e$. Thus, we obtain a factor of $\nu_k^S\nu_{k'}^S$. However, any choice at the face labelled $k$ gives a graph in the automorphism class, which yields another factor of $\frac{1}{\nu_k^S}$. Thus, for any new face we attach to $k'$, we obtain a factor of $\nu_{k'}^S$. By proceeding as such through the contruction, we obtain the lemma.
\end{proof}

We now show that given $I\subset[\ell(\mu)],J\subset[\ell(\nu)]$, $\mu'\preceq_I\mu,\nu'\preceq_J\nu$, $S\in\mathcal{S}((\mu',\nu'),(\mu,\nu))$ and a Hurwitz galaxy $\Gamma'$ of type $(g,\mu',\nu')$ our construction in terms of gluing sequences encoded by forests yields $\mathrm{mult}_{I,J}(S)$ many Hurwitz galaxies of type $(g,\mu,\nu)$, which we construct from $\Gamma'$. In particular it only depends of the data $g$ and $\mu'\preceq_I\mu,\nu'\preceq_J\nu$ -- not on the specific Hurwitz galaxy of type $(g,\mu',\nu')$ we choose.\par 
First we see, that the number of Hurwitz galaxies of type $(g,\mu,\nu)$ we obtain is given by the number of feasible families of forests weighted by the choices given in \cref{lem:enum}. Note, that given $S$ the expression in \cref{equ:multface} only depends on the valency of the vertices. Thus, this is all the data, we need to fix. We use the following general fact.

\begin{fact}
We fix $a_1,\dots, a_n$, such that $\sum a_i=n-k$. The number of forests on a vertex set labeled by $[n]$, such that $\mathrm{val}(i)=a_i$ for $i=1,\dots,k$, $\mathrm{val}(i)=a_i+1$ for $i=k+1,\dots,n$ and such that the vertices $1,\dots, k$ are in different components is
\begin{equation}
\sum_{l\in[k]}\binom{n-k-1}{a_1,\dots,a_{l-1},a_l-1,a_{l+1},\dots,a_{n}}.
\end{equation}
\end{fact}

As for any choice of such a forest, we obtain a different graph, we combine this fact with \cref{lem:enum} and obtain a factor of 
\begin{equation}
\sum_{\substack{\underline{a}\in\mathbb{Z}^{\mu_i^S+|I^S_i|}:\\
|\underline{a}|=|I^S_i|}}\sum_{k\in [\mu_i^S]}\binom{|I^S_i|-1}{a_1,\dots,a_{k-1},a_k-1,a_{k+1},\dots,a_{\mu_i^S+|I^S_i|}}\prod_{p\in I^S_i}s_p^{a_p},
\end{equation}
which coincides with $\mathrm{mult}_{S;I,J}^{\bullet}(i)$. Further, we have to choose the edge labels of the forest. By a subset of $[2g-2+\ell(\mu)+\ell(\nu)]$ of size $2g-2+\ell(\mu')+\ell(\nu')$, we determine the labels of $4-$valent vertices of $\Gamma'$ (as the linear order of the labels must be preserved). This gives an additional factor of
\begin{equation}
\binom{2g-2+\ell(\mu)+\ell(\nu)}{2g-2+\ell(\mu')+\ell(\nu')}.
\end{equation}
Furthermore, the only condition for remaining labels of $4-$valent vertices (which corresponds to the edge labels in the forests) is that each one appears exactly ones. There are $(2g-2+\ell(\mu)+\ell(\nu))-(2g-2+\ell(\mu')+\ell(\nu'))=|I^c|+|J^c|$ many edge labels left, which gives a factor of 
\begin{equation}
(|I^c|+|J^c|)!.
\end{equation}

Combining these considerations, we obtain a factor of
\begin{equation}
\binom{2g-2+\ell(\mu)+\ell(\nu)}{2g-2+\ell(\mu')+\ell(\nu')}\cdot (|I^c|+|J^c|)!\cdot \prod_{i\in\ell(\mu)}\mathrm{mult}_{S}(i)\prod_{j\in\ell(\nu)}\mathrm{mult}_{S}(j),
\end{equation}
which coincides with $\mathrm{mult}_{I,J}(S)$. This way, we obtain the first summand in \cref{thm:main1}.\par 
The second part corresponds to constructing Hurwitz galaxies from the empty Hurwitz galaxy. We will explain this in more detail in \cref{proof2}. For now, we note that for any choice $j\in[\ell(\nu)]$ and $1\le a\le\mathrm{min}(\mu_1,\nu_j)$, there is a unique Hurwitz galaxy $G_a$ of type $(0,a,a)$. We now observe that by the same considerations as above the summand
\begin{equation}
\sum_{j\in[\ell(\nu)]}\sum_{a=1}^{\mathrm{min}(\mu_1,\nu_j)}\sum_{S\in\mathcal{S}_{\{1\},\{j\}}((a,a),(\mu,\nu))}\mathrm{mult}_{\{1\},\{j\}}(S)
\end{equation}
yields the number Hurwitz galaxies $G$ of type $(0,\mu,\nu)$ obtained by consecutively gluing loop faces into $G_a$, such that the white face (black face) of $G_a$ is labeled $1$ (resp. $j$) in any $G$ we obtain.\par 
It remains to prove that those are all the graph we obtain.

\subsection{From pruning to gluing}
\label{proof2}
We now analyse the pruning process and explain, how to obtain a gluing sequence and a feasible family of forests from this process.\par 
Let $\Gamma$ be a Hurwitz galaxy of type $(g,\mu,\nu)$ and let $E_i$ be a white loop face of $\Gamma$ labelled $i$ for some $i\in[\mu]$. We note that $E_i$ is adjacent to exactly one black face $E^j$ labelled $j$ for some $j\in[\nu]$. When we remove $E_i$ from $\Gamma$ we obtain a new Hurwitz galaxy $\Gamma'$ of type $(g,\mu',\nu')$, where $\mu'\preceq_{[\mu]\backslash\{i\}}\mu$ and $\nu'\preceq_{[\nu]}\nu$ are given by 
\begin{itemize}
\item $\mu_k'=\mu_k$ for $k\in[\mu]\backslash\{i\}$,
\item $\nu'_k=\nu_k$ for $k\neq j$,
\item $\nu'_j=\nu_j-\mu_i$.
\end{itemize}
Thus, we see that $(\mu,\nu)$ can be obtained from $(\mu',\nu')$ by a white gluing step of type
\begin{equation}
(\mu',\nu',\mu,\nu,[\mu]\backslash\{i\},[\nu],i,j,\mu_i)^{\circ}.
\end{equation}
This corresponds to reversing the removal of $E_i$, i.e. gluing the face $E_i$ into the face $E^j$.\par

We proceed similarly for black faces: For a black loop face $E^j$ of $\Gamma$ labelled $j$ from $j\in[\nu]$, we note that $E^j$ is adjacent to only one white face $E_i$ for some $i\in[\mu]$. Removing $E^j$ we obtain a new Hurwitz galaxy of type $(g,\mu',\nu')$, where $\mu'\preceq_{[\mu]}\mu$ and $\nu'\preceq_{[\nu]\backslash\{j\}}$, such that
\begin{itemize}
\item $\mu'_k=\mu_k$ for $k\neq i$,
\item $\mu'_i=\mu_i-\nu_j$,
\item $\nu'_k=\nu_k$ for $k\in[\nu]\backslash\{j\}$.
\end{itemize}

Thus, we see that $(\mu,\nu)$ can be obtained from $(\mu',\nu')$ by a black gluing step of type
\begin{equation}
(\mu',\nu',\mu,\nu,[\mu]\backslash\{i\},[\nu],i,j,\mu_i)^{\bullet}.
\end{equation}
This corresponds to reversing the removal of $E^j$, i.e. gluing the face $E^j$ into the face $E_i$. Thus, while we proceed through \cref{con:prun}, collecting the gluing steps as above, we obtain a gluing sequence $S$.\par 
The next step is defining a feasible family of forests associated to the pruning process. Let $\Gamma$ be a Hurwitz galaxy of type $(g,\mu,\nu)$. We focus on the white faces, as the case for black faces is completely analogous.\par 
We fix $i\in[\ell(\mu)]$, and consider the black faces inside the face labelled $i$, which are removed in the pruning process. We collect their labels in the set $B_i^{\Gamma}$. Note, that when all faces in $B_i^{\Gamma}$ are removed the perimeter of the face labelled $i$ is equal to $\mu_i^S$. We call this face the \textit{underlying pruned face} of $i$. We divide the underlying pruned face of $i$ into $\mu_i^S$ segments as before and define a graph $F_i$ on the vertex set $[\mu_i^S]\sqcup B_i^{\Gamma}$ by the following adjacency rules
\begin{itemize}
\item The vertex $j\in B_i^{\Gamma}$ is adjacent to $k\in[\mu_i^S]$ if the face labelled $j$ shares a vertex with the segment labelled $k$ of the underlying pruned face labelled $i$.
\item Two vertices $j,k\in B_i^{\Gamma}$ are adjacent if the corresponding faces share a vertex. We label the edge by the label of the vertex.
\end{itemize}
Since the faces in $B_i^{\Gamma}$ are removed in the pruning process the resulting graph cannot have any circles. Thus, it is a forest, which -- by construction -- has $\mu_i^S$ connected components, which each contain one vertex with label in $[\mu_i^S]$.\par 
By performing this process for all faces in $\Gamma$, we obtain a feasible family of forests.\par 
Moreover, we see that $\Gamma$ is obtained from $\Gamma'$ by applying the construction discussed in \cref{proof1} with data given by gluing sequence $S$ and the family of feasible forests $((F_i)_{i\in [\ell(\mu)]},\allowbreak(F_j)_{j\in [\ell(\nu)]})$. This follows immediately, once we observe that $B_i^{\Gamma}=B_i^S$.\par 
Note, that in the above considerations, we implicitely assumed that the underlying pruned Hurwitz galaxy of $\Gamma$ is not the empty graph. We now consider these cases. The pruning process can only yield the empty graph whenver $g=0$, due to the Riemann Hurwitz condition as there are always at least two $4-$valent vertices when $g\ge 1$ (i.e. $2g-2+\ell(\mu)+\ell(\nu)\ge 2$ for $g\ge 1$). Thus, let $g=0$ and consider a Hurwitz galaxy of type $(0,\mu,\nu)$, such that the pruning process yields the empty Hurwitz galaxy. In this case, we can assume that the white face labeled $1$ is the last one to be removed, as the pruning process is symmetric. Thus, we obtain a Hurwitz galaxy $G_a$ of type $(0,a,a)$, where $a\le\mathrm{min}(\mu_1,\nu_j)$, where $j$ is the last black face removed simultaneously with the white face labeled $1$. Thus, we proceed as above, but collect the data concerning the gluing sequences and families of forests with respect with respect to the underlying Hurwitz galaxy $G_a$. This completes the proof.

We finish with an example of associating a gluing sequence and a feasible family of forests to the pruning process of a Hurwitz galaxy.

\begin{example}
We consider the graph in \cref{ex:prun} and pruning process illustrated in \cref{fig:ex}. We remove three faces, thus the resulting gluing sequence consists of three gluing steps $S=(\tau_1,\tau_2,\tau_3)$. The first step in the pruning process removes a white face of perimeter $1$ labeled $1$ from a black face of perimeter $3$ labeled $1$. This yields $\tau_3=((6,1,1),(2,2,1,3),(6,1,1),(3,2,1,3),\allowbreak\{2,3,4\},[4],1,1,1)^{\circ}$. Similarly, we obtain $\tau_2=((5,1,1),(2,2,3),(6,1,1),(3,2,1,3),\{2,3,4\},\{1,2,\allowbreak4\},2,3,1)^{\bullet}$ and $\tau_1=((3,1,1),(2,3),(6,1,1),(3,2,1,3),\{2,3,4\},\{2,4\},2,1,2)^{\bullet}$.\par 
Moreover, the forests for each face are illustrated in \cref{fig:last}.
\end{example}
\begin{figure}
\begin{center}
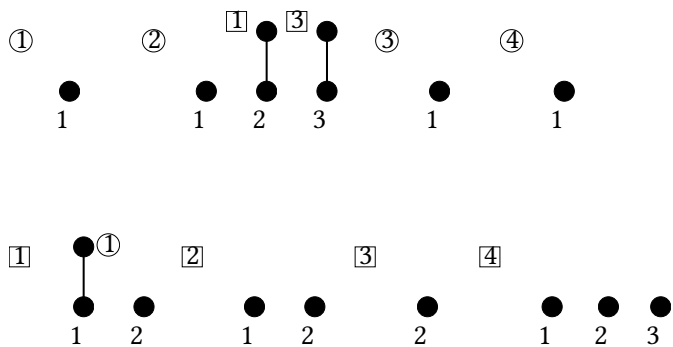
\caption{At the top of the figure the forest for each white face, on the bottom the forest for each black face in the pruning process in \cref{fig:ex}.}
\label{fig:last}
\end{center}
\end{figure}

\printbibliography
\end{document}